\documentclass[11pt]{article}

\makeindex
\usepackage[all]{xy}
\usepackage{graphicx}

\usepackage{amsmath}
\usepackage{amsfonts}
\usepackage{amssymb}
\usepackage{amsthm}

\DeclareMathOperator{\Ann}{{\rm Ann}}
\newcommand{\isom}{\cong}

\newcommand{\Z}{{\bf{Z}}}

\newcommand{\Q}{{\bf{Q}}}
\newcommand{\Qbar}{{\overline{\Q}}}
\newcommand{\R}{{\bf{R}}}
\newcommand{\C}{{\bf{C}}}
\newcommand{\F}{{\bf{F}}}
\newcommand{\T}{{\bf{T}}}

\newcommand{\m}{{\mathfrak{m}}}

\newcommand{\Mid}{|} 
\newcommand{\lm}{\Big|}
\newcommand{\miD}{|}
\newcommand{\rmid}{\Big|}
\newcommand{\divs}{\!\mid\!}
\newcommand{\ndiv}{\!\nmid\!}
\newcommand{\tensor}{\otimes}

\newcommand{\ra}{{\rightarrow}}
\newcommand{\lra}{{\longrightarrow}}

\newcommand{\hra}{\hookrightarrow}

 1
\DeclareFontEncoding{OT2}{}{} 
  \newcommand{\textcyr}[1]{%
    {\fontencoding{OT2}\fontfamily{wncyr}\fontseries{m}\fontshape{n}%
     \selectfont #1}}

\newcommand{\Sha}{{\mbox{\textcyr{Sh}}}} 
\newcommand{\Shan}{{\Mid \Sha(A_f) \miD_{\rm an}}}

\newcommand{\ce}{c_{\scriptscriptstyle{E}}}

\newcommand{\Edual}{E^{\vee}}
\newcommand{\Fdual}{F^{\vee}}
\newcommand{\Ihat}{{\widehat{I}}}

\newcommand{\Hom}{{\rm Hom}}

\newcommand{\ord}{{\rm ord}}
\newcommand{\LAf}{{L_{\scriptscriptstyle{E}}}}
\newcommand{\OAf}{{\Omega_{\scriptscriptstyle{E}}}}

\newcommand{\NerA}{{\mathcal{E}}}

\newcommand{\annT}{{{\rm Ann_\T}}}

\newcommand{\Jmd}{{(J'_{\m'})^\vee}}

\newcommand{\comment}[1]{}

\newtheorem{lem}{Lemma}[section]
\newtheorem{cor}[lem]{Corollary}
\newtheorem{prop}[lem]{Proposition}
\newtheorem{conj}[lem]{Conjecture}
\newtheorem{thm}[lem]{Theorem}

\theoremstyle{definition}

\newtheorem{rmk}[lem]{Remark}

\newcommand{\thetitle}
{Visibility and the Birch and Swinnerton-Dyer conjecture
for analytic rank zero}

\begin{document}
\parindent=2em

\title{\thetitle}
\author{Amod Agashe
\footnote{This material is based upon work supported by the National Science 
Foundation under Grant No. 0603668.}}
\date{}
\maketitle

\begin{abstract}
Let $E$ be an optimal elliptic curve over~$\Q$ of conductor~$N$
having analytic rank zero, i.e., 
such that the $L$-function~$L_E(s)$ of~$E$ does not vanish at~$s=1$.
Suppose there is another optimal 
elliptic curve over~$\Q$ of the same conductor~$N$
whose Mordell-Weil  rank is greater than zero and
whose associated newform is congruent to the newform associated to~$E$
modulo an integer~$r$.
The theory of visibility 
then shows that under certain additional hypotheses,
$r$ divides the product of the 
order of the Shafarevich-Tate group of~$E$ 
and the orders of the arithmetic component groups of~$E$.
We extract an explicit integer factor
from the the Birch and Swinnerton-Dyer
{\em conjectural} formula for the product mentioned above, and 
under some hypotheses similar to the ones made in the situation above,
we show that
$r$ divides this integer factor.
This provides theoretical evidence for
the second part of
the Birch and Swinnerton-Dyer conjecture in the analytic rank zero case.
\end{abstract}


\section{Introduction}
\label{sec:bsd}

Let $N$ be a positive integer. Let
$X_0(N)$ be the modular curve over~$\Q$
associated to~$\Gamma_0(N)$, and
let $J=J_0(N)$ denote the Jacobian of~$X_0(N)$, which is
an abelian variety over~$\Q$. 
Let $\T$ denote the Hecke algebra, which is 
the subring of endomorphisms of~$J_0(N)$
generated by the Hecke operators (usually denoted~$T_\ell$
for $\ell \ndiv N$ and $U_p$ for $p\divs N$). 
If $f$ is a newform of weight~$2$ on~$\Gamma_0(N)$, then 
let $I_f = \annT f$ and let $A_f$ denote the associated {\em newform
quotient} $J/I_f J$, which is
an abelian variety over~$\Q$. 
Let $\pi$ denote the quotient map $J \ra J/I_f J = A_f$.
By the {\em analytic rank} of~$f$, we mean the
order of vanishing at $s=1$ of~$L(f,s)$. 
The analytic rank of~$A_f$ is then the analytic rank of~$f$ times the dimension of~$A_f$.
Now suppose that the newform~$f$ has integer Fourier coefficients.
Then $A_f$ is an elliptic curve, and we denote it by~$E$ instead.
Since $E$ has dimension one, its analytic rank is the same
as that of~$f$.

Now suppose that $\LAf(1) \neq 0$ (i.e., $f$ has analytic rank zero). Then 
by~\cite{kollog:finiteness}, $E$ has Mordell-Weil rank zero,
and the Shafarevich-Tate group~$\Sha(E)$ of~$E$ is finite.
Let $\NerA$ denote the N\'eron model of~$E$ over~$\Z$
and let $\NerA^0$ denote the largest open subgroup scheme
of~$\NerA$ in which all the fibers are connected.
Let $\OAf$ denote the volume of~$E(\R)$ with
respect to the measure given by
a generator of
the rank one $\Z$-module of invariant
differentials on~$\NerA$.
If $p$ is a prime number, then the group of~$\F_p$-valued
points of the quotient~$\NerA_{\F_p}/\NerA^0_{\F_p}$ is called
the (arithmetic) component group of~$A$ and its order
is denoted $c_p(A)$.
Throughout this article, we use the 
symbol~$\stackrel{?}{=}$ to denote a conjectural equality.

Considering that $\LAf(1) \neq 0$,
the second part of the Birch and Swinnerton-Dyer conjecture says the following:

\begin{conj}[Birch and Swinnerton-Dyer] \label{bsd2}
\begin{eqnarray} \label{bsdform}
\frac{\LAf(1)}{\OAf} 
\stackrel{?}{=}
\frac {\Mid \Sha(E) \miD \cdot \prod_{\scriptscriptstyle{p |N}}  c_p(E)}
      { \Mid E(\Q) \miD^2}\ .
\end{eqnarray}
\end{conj}

It is known that ${\LAf(1)}/{\OAf}$ is a rational number.
The importance of the second part of
the Birch and Swinnerton-Dyer conjecture is that it gives a conjectural
value of~$\Mid \Sha(E) \miD$ in terms of the other quantities
in~(\ref{bsdform}) (which can often be computed). Let us denote this 
conjectural value of~$\Mid \Sha(E) \miD$ 
by~$\Mid \Sha(E) \miD_{\rm an}$ (where ``an'' stands for ``analytic'').
The theory of
Euler systems has been used to bound~$|\Sha(E)|$ 
from above in terms~$\Shan$ 
as in the work
of Kolyvagin and of Kato (e.g., see \cite[Thm $8.6$]{rubin:eulerell}).
Also, the Eisenstein series method 
is being used by Skinner-Urban (as yet unpublished)
to try to show that $\Shan$ divides~$|\Sha(E)|$.
In both of the methods above,
one may have to stay away from certain primes. 

The conjectural formula~(\ref{bsdform}) may be rewritten as follows:
\begin{eqnarray} \label{bsdformula}
|E(\Q)|^2 \cdot \frac{\LAf(1)}{\OAf}
\stackrel{?}{=}
{\Mid \Sha(E) \miD \cdot \prod_{\scriptscriptstyle{p |N}}  c_p(E)} \ .
\end{eqnarray}
We shall refer to the formula above as the Birch and Swinnerton-Dyer
conjectural formula.

Now suppose that $f$ is congruent modulo a prime~$p$
to another newform~$g$ that has
integer
Fourier coefficients and whose associated elliptic curve has
positive Mordell-Weil rank.  
Let $r$ denote
the highest power of~$p$ modulo which this congruence holds.
Then the theory of visibility 
(e.g., as in~\cite{cremona-mazur})
often shows
that $r$ divides ${\Mid \Sha(E) \miD \cdot \prod_{\scriptscriptstyle{p |N}}  c_p(E)}$,
the right side of the Birch and Swinnerton-Dyer conjectural  formula~(\ref{bsdformula});
we give precise results along these lines
in Section~\ref{sec:vis}. When this happens,
the conjectural formula~(\ref{bsdformula}) says that $r$ should also divide
the left side of~(\ref{bsdformula}), which is
$|E(\Q)|^2 \cdot \frac{\LAf(1)}{\OAf}$ (since it is not known that
the rational number $|E(\Q)|^2 \cdot \frac{\LAf(1)}{\OAf}$ is an integer,
what we mean here and henceforth
is that the order at~$p$ of this rational number is at least
${\ord_p} r$).
In Section~\ref{sec:visfac}, we show that this does happen
under somewhat similar hypotheses.
In Section~\ref{sec:proofs}, we give the proof of our main result 
(Theorem~\ref{thm:main}); 
in the proof, we actually
extract an explicit integer factor from 
$|E(\Q)|^2 \cdot \frac{\LAf(1)}{\OAf}$,
and under certain hypotheses,
we show that  $r$ divides this integer factor. 
The reader who is interested in seeing only the precise statements of our 
main results 
may read Sections~\ref{sec:vis} and~\ref{sec:visfac},
skipping proofs. 
In each section, we continue to use the notation introduced
in earlier sections (unless mentioned otherwise).

We remark that the results of this article are very analogous to the results
obtained in~\cite{ag:visrk1}, where we treated the case where $E$
had analytic rank one. We also take the opportunity
to point out some mistakes in~\cite{ag:visrk1} (see Remark~\ref{rmk:correction}).
Finally, our results for
the case where $r = p$ (i.e., if $f$ and~$g$ are
congruent modulo~$p$, but not modulo~$p^2$), are covered to some
extent in~\cite{agmer}. In fact, the present article arose from
our efforts to generalize some of the results in~\cite{agmer}.\\

\noindent {\it Acknowledgements:} We are grateful to M.~Emerton
for pointing out some errors related to the statement of
Lemma~\ref{lem:containment2}
in an earlier version of this article.

\section{Visibility and the right side of the Birch and Swinnerton-Dyer conjectural  formula}
\label{sec:vis}

Let $F$ denote the elliptic curve  associated to the newform~$g$.
If $A$ is an abelian variety, then we denote its dual abelian variety by~$A^\vee$. 
If $h$ is a newform of weight~$2$ on~$\Gamma_0(N)$, then 
by taking the dual of the quotient map $J_0(N) \ra A_h$
and using the self-duality of~$J_0(N)$, we may view
$A_h^\vee$ as an abelian subvariety of~$J_0(N)$. In particular, we may view
$\Edual$ and~$\Fdual$ as abelian subvarieties of~$J_0(N)$.
We say that a maximal ideal~$\m$ of~$\T$
satisfies {\em multiplicity one} if
$J_0(N)[\m]$ is two dimensional over~$\T/\m$.
Consider the following hypothesis on~$p$:\\
(*) if $\m$ is a maximal ideal of~$\T$ with residue characteristic~$p$
and $\m$ is in the support of~$J_0(N)[I_f + I_g]$, 
then $\m$ satisfies multiplicity one.


The following lemma is Lemma~2.1 from~\cite{ag:visrk1}; we repeat the statement
here since we shall refer to it several times.
\begin{lem} \label{lem:multone}
Suppose $p$ is odd, and either\\
(i) $p \ndiv N$ or\\
(ii) $p || N$ and $E[p]$ or $F[p]$ is irreducible. \\
Then $p$ satisfies hypothesis~(*).
\end{lem}

\begin{prop} \label{prop:agst}
(i) 
Suppose that 
$p$ is coprime to 
$$N \cdot |(J_0(N)/\Fdual)(\Q)_{\rm tor}|
\cdot |F(\Q)_{\rm tor}| 
\cdot \prod_{\scriptscriptstyle{{\ell \mid N}}} \big( c_\ell(F) 
\cdot c_\ell(E) \big).$$
Then $r$ 
divides ${\Mid \Sha(E) \miD}$.
If moreover we assume the parity conjecture, then
$r^2$ 
divides ${\Mid \Sha(E) \miD}$.
\\
(ii) Suppose that $p$ is odd,  
that $E[p]$ and~$F[p]$ are irreducible, and that
$p$ does not divide 
$$N \cdot |(J_0(N)/\Fdual)(\Q)_{\rm tor}| \cdot |F(\Q)_{\rm tor}|.$$
Then $p$ 
divides ${\Mid \Sha(E) \miD \cdot
\prod_{\scriptscriptstyle{p |N}}  c_p(E)}$, 
the right hand side of
the Birch and Swinnerton-Dyer formula~(\ref{bsdformula}).
If we assume moreover that $f$ is not congruent modulo
a prime ideal over~$p$ to a newform
of a level dividing~$N/\ell$ for some prime~$\ell$ that divides~$N$
(for Fourier coefficients of index coprime to~$Np$),
and that either $p \nmid N$ or for all primes $\ell$ that divide~$N$,
$p \nmid (\ell-1)$, then $p$ 
divides~$\Mid \Sha(E/\Q) \miD$.
\end{prop}
\begin{proof}
The proof is similar to the proof of Proposition~3.1 in~\cite{ag:visrk1}.
Both parts of the proposition above
follow essentially from Theorem~3.1 of~\cite{agst:vis}, which uses
the theory of visibility.
For Part~(i), take $A = \Edual$, $B = \Fdual$, and $n=r$ 
in~\cite[Thm.~3.1]{agst:vis}, and
note that 
$\Fdual[r] \subseteq \Edual$ by 
Lemma~\ref{lem:multone} and  \cite[Lemma~2.2]{ag:visrk1}
(for the application of~Lemma~\ref{lem:multone}, note 
that $p\nmid N$ by hypothesis, and 
for the application of Lemma~2.2 of~\cite{ag:visrk1},
note that the analytic ranks of~$f$ and~$g$ do not play
any role in the proof of Lemma~2.2 of loc. cit.). 
Then \cite[Thm.~3.1]{agst:vis} says that there is
a map $\Fdual(\Q) / r\Fdual(\Q) \ra \Sha(\Edual)$,
whose kernel has order at most~$r$ raised to the power
the Mordell-Weil rank of~$E$.
Since $F$ has Mordell-Weil rank higher than~$E$, we see that
$r$ divides ${\Mid \Sha(E) \miD}$.
Since $r$ is odd, and $f$ and~$g$ are congruent modulo~$r$,
we see that $f$ and~$g$ have the same eigenvalue
under the Atkin-Lehner involution, and hence the same sign
in their functional equations. 
Thus if we assume the parity
conjecture, then $\Fdual$ has Mordell-Weil rank 
at least that of~$E$ plus two, and so
by the discussion involving \cite[Thm.~3.1]{agst:vis} above,
$r^2$ divides ${\Mid \Sha(E) \miD}$. This proves Part~(i).

For the first statement in Part~(ii), take $A = \Edual$, $B = \Fdual$, and $n=p$ 
in~\cite[Thm.~3.1]{agst:vis}, and note that 
the congruence of $f$ and~$g$ modulo~$p$
forces $\Fdual[p] = \Edual[p]$ 
by~\cite[Thm.~5.2]{ribet:modreps} (cf.~\cite[p.~20]{cremona-mazur}).
For the second statement in Part~(ii), note 
that the additional hypotheses imply that 
$p$ does not divide~$c_\ell(E)$ or $c_\ell(F)$ for any prime~$\ell$
that divides~$N$, as we now indicate.
By~\cite[Prop.~4.2]{emerton:optimal}, if $p$ divides~$c_\ell(E)$ 
for some prime~$\ell$ that divides~$N$, then for some maximal ideal~$\m$
of~$\T$ having characteristic~$p$ and containing~$I_f$, 
either $\rho_\m$ is finite or reducible
(here, $\rho_\m$ is the canonical 
two dimensional representation associated to~$\m$, e.g., as 
in~\cite[Prop.~5.1]{ribet:modreps}).
Since $E[p]$ is irreducible, this can happen only if 
$\rho_\m$ is finite. But 
this is not possible by~\cite[Thm.~1.1]{ribet:modreps},
in view of the hypothesis that  $f$
is not congruent modulo~$p$ to a newform
of a level dividing~$N/\ell$ for any prime~$\ell$ that divides~$N$
(for Fourier coefficients of index coprime to~$Np$),
and either $p \nmid N$ or for all primes $\ell$ that divide~$N$,
$p \nmid (\ell-1)$. Thus $p$ does not divide~$c_\ell(E)$ for any prime~$\ell$
that divides~$N$. Similarly,
$p$ does not divide~$c_\ell(F)$ for any prime~$\ell$
that divides~$N$,
considering that the hypothesis that $f$
is not congruent modulo~$p$ to a newform
of a level dividing~$N/\ell$ for any prime~$\ell$ that divides~$N$
(for Fourier coefficients of index coprime to~$Np$) applies
to $g$ as well, since
$g$ is congruent to~$f$ modulo~$p$. This finishes the proof
of the proposition.
\end{proof}

See~\cite{cremona-mazur} or~\cite{agst:bsd} for examples where
the theory of visibility proves the existence
of non-trivial elements of  the Shafarevich-Tate group
of an elliptic curve of analytic rank zero. 

\section{Congruences and the left side of the Birch and Swinnerton-Dyer conjectural  formula}
\label{sec:visfac}

Considering that under certain hypotheses, the theory of visibility
(more precisely Proposition~\ref{prop:agst}(i)) implies that $r$ 
divides ${\Mid \Sha(E) \miD}$, 
which divides the right hand side of
the Birch and Swinnerton-Dyer conjectural formula~(\ref{bsdformula}),
under similar hypotheses, one should be able to show that $r$ also divides
$|E(\Q)|^2 \cdot \frac{\LAf(1)}{\OAf}$, which is 
the left hand side of~(\ref{bsdformula}).
The theory of Euler systems says under certain hypotheses
that the order of~$\Sha(E)$ divides its Birch and Swinnerton-Dyer conjectural
order (e.g., as in the work of Kolyvagin and Kato). Thus,
in conjunction with Proposition~\ref{prop:agst}(i), the theory of Euler
systems shows that under certain additional hypotheses, 
$r$ does divides $|E(\Q)|^2 \cdot \frac{\LAf(1)}{\OAf}$. 
For example, we have the following:
\begin{prop} \label{prop:euler}
Suppose that 
$p$ is coprime to 
$$2 \cdot N \cdot |(J_0(N)/\Fdual)(K)_{\rm tor}|
\cdot |F(K)_{\rm tor}| 
\cdot \prod_{\scriptscriptstyle{{\ell \mid N}}} \big( c_\ell(F) 
\cdot c_\ell(E) \big).$$
Assume that the image of the absolute
Galois group of~$\Q$ acting on~$E[p]$ is isomorphic to~${\rm GL}_2(\Z/p\Z)$.
Then $r$ divides 
$|E(\Q)|^2 \cdot \frac{\LAf(1)}{\OAf}$
and the Birch and Swinnerton-Dyer conjectural
order of~$\Sha(E)$.
If moreover we assume the parity conjecture, then
$r^2$ divides
$|E(\Q)|^2 \cdot \frac{\LAf(1)}{\OAf}$ and
the Birch and Swinnerton-Dyer conjectural
order of~$\Sha(E)$.
\end{prop}
\begin{proof}
Proposition~\ref{prop:agst}(i), which uses the
theory of visibility, implies that $r$
divides ${\Mid \Sha(E) \miD}$,
and that $r^2$
divides ${\Mid \Sha(E) \miD}$ if we assume the parity conjecture.
The result now follows by~\cite[Theorem~13]{stein-wuthrich},
which uses the theory of Euler systems and is an extension of a theorem of Kato.
\end{proof}

The pullback of a generator of
the rank one $\Z$-module of invariant
differentials on the N\'eron model of~$E$ to~$X_0(N)$ (under
the modular parametrization) is a multiple
of the differential $2 \pi i f(z) dz$ by a rational number;
this number is
called the Manin constant of~$E$, and we
shall denote it by~$c_{\scriptscriptstyle E}$.
It is conjectured that ~$c_{\scriptscriptstyle E}$ is one,
and one knows that $c_{\scriptscriptstyle E}$ 
is an integer, and that if $p$ is a prime such that $p^2 \nmid 4 N$, 
then $p$ does not divide~$c_{\scriptscriptstyle E}$
(by~\cite[Cor.~4.1]{mazur:rational} 
and~\cite[Thm.~A]{abbes-ullmo}).

\begin{thm} \label{thm:main} 
Suppose that $p$ is odd and satisfies the hypothesis~(*).
Assume that $f$ and~$g$ are not congruent modulo
a prime ideal over~$p$ to any other newforms of level dividing~$N$
(for Fourier coefficients of index coprime to~$Np$).
Suppose that either $p^2 \ndiv N$ or that the Manin 
constant~$\ce$ is one (as is conjectured).
Then $r^2$ divides $|E(\Q)|^2 \cdot \frac{\LAf(1)}{\OAf}$,
the left side of the Birch and Swinnerton-Dyer conjectural  formula~(\ref{bsdformula}).
\end{thm}

We shall prove this theorem in Section~\ref{sec:proofs}.

\begin{cor}\label{cor1}
Suppose that $p$ is odd, 
that $f$ and~$g$ are not congruent modulo
a prime ideal over~$p$ to any other newforms of level dividing~$N$
(for Fourier coefficients of index coprime to~$Np$),
and that either\\
(a) $p \ndiv N$ or\\
(b) $p || N$ and $E[p]$ or $F[p]$ is irreducible. \\
Then $r^2$ divides $|E(\Q)|^2 \cdot \frac{\LAf(1)}{\OAf}$
and the Birch and Swinnerton-Dyer conjectural
order of~$\Sha(E)$.
\end{cor}
\begin{proof} 
The statement that $r^2$ divides $|E(\Q)|^2 \cdot \frac{\LAf(1)}{\OAf}$
follows from the theorem above,
considering
that the hypothesis~(*) is satisfied, in view of Lemma~\ref{lem:multone}.
By the hypothesis
that $f$ and~$g$ are not congruent modulo
a prime ideal over~$p$ to any other newforms of level dividing~$N$
(for Fourier coefficients of index coprime to~$Np$),
as explained in the proof of Proposition~\ref{prop:agst}(ii),
$p$ does not divide~$c_\ell(E)$ for any prime~$\ell$.
Hence, by~(\ref{bsdformula}), 
$r^2$ divides the Birch and Swinnerton-Dyer conjectural
order of~$\Sha(E)$. 
\end{proof}

In view of Proposition~\ref{prop:agst},
Corollary~\ref{cor1} 
provides theoretical evidence towards 
the Birch and Swinnerton-Dyer conjectural formula~(\ref{bsdformula}).
Corollary~\ref{cor1} may also be compared to the similar 
Proposition~\ref{prop:euler} 
that uses the theory of visibility and the theory
of Euler systems. Note that in Corollary~\ref{cor1}, we do not
assume the following hypotheses of Proposition~\ref{prop:euler}:
$p \ndiv N$ (although we do need that $p^2 \ndiv N$),
$p$ does not divide $|F(K)_{\rm tor}| 
\cdot \prod_{\scriptscriptstyle{{\ell \mid N}}} 
\big( c_\ell(F) 
\cdot c_\ell(E) \big)$,
and the image of the absolute
Galois group of~$\Q$ acting on~$E[p]$ is isomorphic to~${\rm GL}_2(\Z/p\Z)$.
Moreover, Corollary~\ref{cor1} gives the stronger conclusion that
the {\em square} of~$r$ 
divides $|E(\Q)|^2 \cdot \frac{\LAf(1)}{\OAf}$
and the Birch and Swinnerton-Dyer conjectural
order of~$\Sha(E)$ without assuming the parity conjecture.
However, in Corollary~\ref{cor1}, we do have the extra hypothesis
that 
$f$ and~$g$ are not congruent modulo
a prime ideal over~$p$ to any other newforms of level dividing~$N$
(for Fourier coefficients of index coprime to~$Np$). 
This hypothesis is used only via Lemma~\ref{lem:containment2},
and so if it could be removed from that lemma, then it can be
removed from Theorem~\ref{thm:main} and Corollary~\ref{cor1}.
In any case, our proof of Theorem~\ref{thm:main} does not use
the theory of visibility or the theory of Euler systems, 
and is much more elementary than
either theories. In fact, our approach may be considered an alternative
to the theory of Euler systems in the context where the theory of visibility
predicts non-triviality of Shafarevich-Tate groups for analytic rank zero.

\section{Proof of Theorem~\ref{thm:main}} \label{sec:proofs}

We work in slightly more generality in the beginning
and assume that $f$ and~$g$ are any newforms 
(whose Fourier coefficients need not be integers),
with $f$ having analytic rank zero and $g$ having
analytic rank greater than zero. 
Thus the associated newform quotients~$A_f$ and~$A_g$ need not be elliptic
curves, but we will still denote them by~$E$ and~$F$ (respectively)
for simplicity of notation.

Recall that $I_g = {\rm Ann}_{\scriptscriptstyle \T} g$. Let
$J' = J/(I_f \cap I_g)J$ and let $\pi''$ denote the quotient map 
$J \ra J'$. Then the quotient
map $J \stackrel{\pi}{\ra} E$ factors through~$J'$; let
$\pi'$ denote the map $J' \ra E$ in
this factorization. Let $F'$ denote the kernel of~$\pi'$. 
Let $E'$ denote the image of~$\Edual \subseteq J$ in~$J'$ under
the  quotient map~\mbox{$\pi'': J \ra J'$}.
Let $B$ denote the kernel of the projection map~$\pi: J \ra E$; it is the
abelian subvariety~$I_f J$ of~$J$.
We have the following diagram, in which the two sequences of four arrows are 
exact (one horizontal and one upwards diagonal):

$$\xymatrix{
 & \Fdual \ar@{^(->}[rd] & \quad \Edual \ar@{^(->}[d] \ar[rd]^{\sim} & & 0 \\
0 \ar[r] & B \ar[dd] \ar[r] & J \ar[r]^{\pi} \ar[d]_{\pi''} & E \ar[r] 
\ar[ur] & 0\\
& & J'\ar[rd] \ar[ur]_{\pi'} &  & \\
& F' \ar[ur] & & F  & \\
0 \ar[ur] & & & & \\
}$$

Now $F'$ is connected, since it is
a quotient of~$B$ (as a simple diagram chase above shows) and $B$ is connected. 
Thus, by looking at dimensions, one sees that $F'$ is the image of~$\Fdual$
under~$\pi''$. Since the composite $\Fdual \hookrightarrow J \ra J' \ra F$
is an isogeny, 
the quotient map $J' \ra F$ induces an isogeny
$\pi''(\Fdual) \sim F$, and hence 
an isogeny $F' \sim F$. 
Let $E'$ denote $\pi''(\Edual)$. Since $\pi$ induces an isogeny
from~$\Edual$ to~$E$, we see that $\pi'$ also induces an isogeny
from~$E'$ to~$E$. 

Let $\Im$ denote the annihilator, under the action of~$\T$,
of the divisor $(0) - (\infty)$,
considered as an element of~$J_0(N)(\C)$. 
We have an isomorphism
$$H_1(X_0(N),{\Z}) \tensor {\R} \stackrel{\cong}{\lra}
{\rm Hom}_{{\C}} (H^0(X_0(N), \Omega^1),{\C}),$$ obtained
by integrating differentials along cycles (see~\cite[\S~IV.1]{lang:modular}).
Let $e$ be the element of~$H_1(X_0(N),{\Z}) \tensor {\R}$ that corresponds
to the map $\omega \mapsto - \int_{\{0,i\infty\}} \omega$ under this
isomorphism. It is called the {\em winding element}.
By~\cite[II.18.6]{mazur:eisenstein}, we have $\Im e \subseteq H_1(X_0(N),\C)
= H_1(J_0(N),\C)$
(note that in loc. cit., the definition of $\Im$ is different and $N$ 
is assumed to be prime; but the only essential property of $\Im$ that
is used in the proof is that $\Im$ annihilates the divisor $(0) - (\infty)$,
and the assumption that $N$ is prime is not used). 
If $\phi$ is a map of abelian varieties over~$\Q$, then we denote
the induced map on the first homology groups by~$\phi_*$.

\begin{lem} \label{lem:Eiscont}
$\pi''_*(\Im e) \subseteq H_1(E', \Z)$.
\end{lem}
\begin{proof}
Since $J'$ is isogenous to $E' \oplus F'$, we have
$H_1(J', \Z) \tensor \Q \isom H_1(E', \Z) \tensor \Q \oplus H_1(F', \Z) \tensor \Q$.
Viewing $\pi''_*(\Im e)$ as a subset of~$H_1(J', \Z) \tensor \Q$, 
it suffices to show that
$\pi''_*(\Im e) \cap (H_1(F', \Z) \tensor \Q) = 0$.
Suppose $x \in \pi''_*(\Im e) \cap (H_1(F', \Z) \tensor \Q) $;
we need to show that then $x = 0$. For
some integer~$n$, we have $nx \in H_1(F', \Z)$, and for 
some $t \in \Im$, we have $t \pi''_*(e) = nx$. 
Let $\omega$ be a differential over~$\Q$ on~$F'$, 
which we may view as a differential on~$J'$.
Then $\pi''^*(\omega)$, when viewed as a differential on~$X_0(N)$,
is of the form $2 \pi i h(z) dz$ for 
some $h$ in~$S_2(\Gamma_0(N), \Q)[I_g]$.
Thus 
$\int_{nx} \omega = 
\int_{t \pi''_*(e)} \omega = 
\int_{te} 2 \pi i h(z) dz = 
\int_{e} 2 \pi i (t h)(z) dz$.
Now $th \in  S_2(\Gamma_0(N), \Q)[I_g]$, and 
so $th$ is a $\Q$-linear combination of the Galois conjugates of~$g$.
Hence 
$\int_{e} 2 \pi i (t h)(z) dz$ is a 
$\Q$-linear combination of 
of $\int_{e} 2 \pi i g^\sigma(z) dz = L(g^\sigma,1)$ for various conjugates
$g^\sigma$ of~$g$, where
$\sigma \in {\rm Gal}(\Qbar/\Q)$. 
Since $g$ has positive analytic rank, $L(g,1) = 0$, and so
$L(g^\sigma,1) = 0$ for all $\sigma \in {\rm Gal}(\Qbar/\Q)$,
e.g., by~\cite[Cor.~V.1.3]{gross-zagier}.
Thus, by the discussion above, we see that 
$\int_{nx} \omega = 0$ for every differential~$\omega$ over~$\Q$ on~$F'$, and
so $nx = 0$ in~$H_1(F', \Z) \tensor \Q$. Hence $x=0$,
as was to be shown.
\end{proof}

There is a complex conjugation involution acting on~$H_1(X_0(N),\C)$,
and if $G$ is a group on which it induces an involution,
then by $G^+$ we mean the subgroup of elements of~$G$ fixed
by the involution. It is easy to see that $e$ is fixed by the
complex conjugation involution, and so
by Lemma~\ref{lem:Eiscont}, we have  $\pi''_*(\Im e) \subseteq H_1(E', \Z)^+$.
The following is an analog of~\cite[Theorem~3.2]{agmer}:

\begin{prop} \label{prop:fact}
Up to a power of~$2$,
\begin{eqnarray}\label{mainform}
\ce \cdot c_\infty(E)\cdot \frac{\LAf(1)}{\OAf} =
\frac{
\Mid \frac{H_1(J', \Z)^+}{H_1(F', \Z)^+ + H_1(E', \Z)^+} \miD \cdot 
\Mid \frac{H_1(E', \Z)^+  + H_1(F', \Z)^+}{\pi''_*(\Im  e) + H_1(F', \Z)^+} \miD}
{\Mid \pi_*(\T  e) / \pi_*(\Im  e) \miD}.
\end{eqnarray}
\end{prop}
\begin{proof}
By~\cite[Thm.~2.1]{agmer}, we have
\begin{eqnarray} \label{formula:lovero}
\frac{\LAf(1)}{\OAf} = 
\frac{\left[ H_1(A_f,{\Z})^+ : \pi_*(\T e) \right]}
     {\ce \cdot c_\infty(E)} \ , 
\end{eqnarray}
where $[ H_1(A_f,{\Z})^+ : \pi_*(\T e) ]$ denotes
the absolute value of the
determinant of an automorphism of~$H_1(A_f,\Q)$ that
takes the lattice $H_1(A_f,{\Z})^+$ isomorphically onto
the lattice $\pi_*(\T e)$. 
Now $\pi''_*$ and $\pi'_*$ are both surjective, since the
kernels of~$\pi''$ and~$\pi'$ (respectively) are connected.
Thus $H_1(E, \Z) = \pi'_*(H_1(J', \Z))$. Putting this in~(\ref{formula:lovero}),
and considering that $\pi''_* (\Im  e) \subseteq H_1(J', \Z)^+$
(since $\Im  e \subseteq H_1(J_0(N), \C)^+$), we get
\begin{eqnarray} \label{inproof}
\ce \cdot c_\infty(E)\cdot \frac{\LAf(1)}{\OAf}  
= [\pi'_*(H_1(J', \Z))^+ : \pi_*(\T  e)]  
= \frac{\Mid \pi'_*(H_1(J', \Z))^+ /\pi'_*(\pi''_* (\Im  e)) \miD}
{\Mid \pi_*(\T  e) / \pi_*(\Im  e) \miD}.
\end{eqnarray}

The long exact sequence of homology associated to the 
short exact sequence~$0 \ra F' \ra J' \ra E \ra 0$ is:
$$\ldots \ra H_1(F',\Z) \ra H_1(J',\Z)
\stackrel{\pi'_*}{\ra} H_1(E,\Z) \ra 0 \ra \ldots$$
Thus $H_1(F',\Z) \subseteq \ker(\pi'_*)$.\\

\noindent{\it Claim:}
$H_1(F',\Z) = \ker(\pi'_*)$.
\begin{proof}
Since $H_1(F',\Z)$ is saturated in~$H_1(J', \Z)$, it suffices
to show that $H_1(F',\Z) \tensor \Q = \ker(\pi'_*) \tensor \Q$,
i.e., that 
the free abelian groups $H_1(F',\Z)$ and~$\ker(\pi'_*)$ have the same rank.
But
\begin{tabbing}
\= ${\rm rank}(\ker(\pi'_*)) = 2 \cdot \dim J' - 2 \cdot \dim E$ \\
\> $= 2 \cdot \dim_\Q S_2(\Gamma_0(N),\Q)[I_f \cap I_g] 
- 2 \cdot \dim_\Q S_2(\Gamma_0(N),\Q)[I_f]$ \\ 
\> $ = 2 \cdot \dim_\Q S_2(\Gamma_0(N),\Q)[I_g] = 
2 \cdot \dim_\Q F' = {\rm rank}(H_1(F',\Z))$.
\end{tabbing} This proves the claim.
\end{proof}

The kernel of the natural map 
$H_1(J', \Z) \ra \pi'_*(H_1(J', \Z))/\pi'_*(\pi''_* (\Im  e))$
is $\ker(\pi'_*) + \pi''_* (\Im  e) = H_1(F', \Z) + \pi''_* (\Im  e)$,
by the claim above
Thus up to a power of~$2$,
\begin{eqnarray} \label{eqn:oddeq}
\Mid \pi_*(H_1(J', \Z))^+ /\pi'_*(\pi''_*(\Im  e)) \miD =
\lm \frac{H_1(J', \Z)^+}{ H_1(F', \Z)^+ + \pi''_*(\Im  e)} \rmid.
\end{eqnarray}

In view of Lemma~\ref{lem:Eiscont},
\begin{eqnarray} \label{eqn:fact}
\lm \frac{H_1(J', \Z)^+}{H_1(F', \Z)^+ + \pi''_*(\Im  e)} \rmid
= \lm \frac{H_1(J', \Z)^+}{H_1(F', \Z)^+ + H_1(E', \Z)^+} \rmid \cdot 
\lm \frac{H_1(E', \Z)^+  + H_1(F', \Z)^+}{\pi''_*(\Im  e) + H_1(F', \Z)^+} \rmid.
\end{eqnarray}
Putting~(\ref{eqn:fact}) in~(\ref{eqn:oddeq}), and 
then putting the result in~(\ref{inproof}), we get
the formula in the proposition.
\end{proof}

\begin{lem} \label{lem:containment2}
Suppose that $p$ satisfies hypothesis~(*), and assume that
$f$ and~$g$ have integer Fourier coefficients (so 
$\Edual$ and~$\Fdual$ are elliptic curves). 
Assume moreover that $f$ and~$g$ are not congruent modulo
a prime ideal over~$p$ to any other newforms of level dividing~$N$
(for Fourier coefficients of index coprime to~$Np$).
Then $E'[r] = F'[r]$, and both are direct summands
of~$E' \cap F'$ as~${\rm Gal}(\Qbar/\Q)$-modules.
\end{lem}
\begin{proof}
By the proof of~\cite[Lemma~2.2]{ag:visrk1},
$(\Edual \cap \Fdual)[p^\infty] = \Edual[r] = \Fdual[r]$.
The kernels of the surjective maps 
$\Edual \ra E'$, $\Fdual \ra F'$, and
$\Edual \cap \Fdual \ra E' \cap F'$
that are induced by~$\pi''$ are all
contained in~$J'^\vee \cap IJ$. 
By~\cite[Theorem~3.6(a)]{ars:moddeg} with $A = J'^\vee$, 
if a prime~$\ell$ divides the order of~$J'^\vee \cap IJ$,
then $\ell$ divides the congruence exponent of~$J'^\vee$
(with notation as in loc. cit.).
By the hypothesis that 
 $f$ and~$g$ are not congruent modulo
a prime ideal over~$p$ to any other newforms of level dividing~$N$
(for Fourier coefficients of index coprime to~$Np$), 
the congruence exponent of~$J'^\vee$ is coprime to~$p$.
Hence the kernels of the maps mentioned
above have orders coprime to~$p$. Thus the maps 
$\Edual \ra E'$, $\Fdual \ra F'$, and
$\Edual \cap \Fdual \ra E' \cap F'$
are all isomorphisms on~$p^n$ torsion points for any
positive integer~$n$ (this can be seen, e.g.,
by the snake lemma applied to the multiplication by~$p^n$ map
on the corrsponding short exact sequence in each situation). In particular the maps
$(\Edual \cap \Fdual)[p^\infty] \ra (E' \cap F')[p^\infty]$,
$\Edual[r] \ra E'[r]$, and $\Fdual[r] \ra F'[r]$ are isomorphisms.
From this and the very first statement in this proof, we see that
$(E' \cap F')[p^\infty] = E'[r] = F'[r]$. The lemma now follows from
the conclusion of the previous sentence.
\end{proof}

\begin{proof}[Proof of Theorem~\ref{thm:main}]
Note that since $F$ has positive Mordell-Weil rank,
$g$ has positive analytic rank (by~\cite{kollog:finiteness}),
and so the discussion of this section applies.
By Proposition~\ref{prop:fact}, we see
that up to a power of~$2$,
\begin{eqnarray} \label{eqn:final}
& &|E(\Q)|^2 \cdot \frac{\LAf(1)}{\OAf}   \\
& =  &\frac{
\Mid \frac{H_1(J', \Z)^+}{H_1(F', \Z)^+ + H_1(E', \Z)^+} \miD \cdot 
\Mid \frac{H_1(E', \Z)^+  + H_1(F', \Z)^+}{\pi''_*(\Im  e) + H_1(F', \Z)^+} \miD}
{\ce \cdot c_\infty(E)} \cdot
\frac{|E(\Q)|}{\Mid \pi_*(\T  e) / \pi_*(\Im  e) \miD}
\cdot |E(\Q)|. \nonumber 
\end{eqnarray}
By Lemma~\ref{lem:containment2}, we see that
$r^2$ divides $|E' \cap F'|$. 
By~\cite[Lemma~4.1]{agmer}, we have
$\Mid \frac{H_1(J', \Z)}{H_1(F', \Z) + H_1(E', \Z)} \miD  = | E' \cap F'| \ $.
Hence $r^2$ divides the term
$\Mid \frac{H_1(J', \Z)^+}{H_1(F', \Z)^+ + H_1(E', \Z)^+} \miD$
on the right side of~(\ref{eqn:final})
(considering that $\Mid \frac{H_1(J', \Z)^+}{H_1(F', \Z)^+ + H_1(E', \Z)^+} \miD$
differs from $\Mid \frac{H_1(J', \Z)}{H_1(F', \Z) + H_1(E', \Z)} \miD$ by
a power of~$2$ and that  $r$ is odd).
The theorem now follows from equation~(\ref{eqn:final}),
in view of the facts that
${\Mid \pi_*(\T  e) / \pi_*(\Im  e) \miD}$ divides
$|E(\Q)|$ (by~\cite[Lemma~3.3]{agmer}), $\ce$ is coprime to~$p$ if $p^2 \ndiv N$
(by~\cite[Cor.~4.1]{mazur:rational}), and $c_\infty(E)$ is a power of~$2$,
hence coprime to~$r$.
\end{proof}

\begin{rmk} \label{rmk:correction}
We would like to take the chance to make some corrections to our
earlier paper~\cite{ag:visrk1}. 
First, the statement of the first part of Proposition~3.1 of loc. cit. should
read: \\
Suppose that 
$p$ is coprime to 
$$N \cdot |(J_0(N)/\Fdual)(K)_{\rm tor}|
\cdot |F(K)_{\rm tor}| 
\cdot \prod_{\scriptscriptstyle{{\ell \mid N}}} \big( c_\ell(F)
\cdot c_\ell(E) \big).$$
Then $r$ 
divides ${\Mid \Sha(E/K) \miD}$.

The proof of the statement above is identical to the proof of Part~(i)
of Proposition~\ref{prop:agst} in this article, with $\Q$ replaced by~$K$.
As a result of this correction, the statement of Proposition~4.1 
of~\cite{ag:visrk1}
should change to:\\
Suppose that 
$p$ is coprime to 
$$2 \cdot N \cdot |(J_0(N)/\Fdual)(K)_{\rm tor}|
\cdot |F(K)_{\rm tor}| 
\cdot \prod_{\scriptscriptstyle{{\ell \mid N}}} \big( c_\ell(F) 
\cdot c_\ell(E) \big).$$
Assume that the image of the absolute
Galois group of~$\Q$ acting on~$E[p]$ is isomorphic to~${\rm GL}_2(\Z/p\Z)$.
Then $r$ divides $|E(K)/ \Z \pi(P)|^2$.

This claim follows from the corrected version of 
Proposition~3.1 mentioned above, and from the paragraph just
after the statement of Theorem 1.1 in~\cite{jetchev-global}.

Finally, in the fourth paragraph of
Section~5 of loc. cit., we claimed that ``since $\Edual[r] = \Fdual[r]$
and both are 
direct summands 
of~$\Edual \cap \Fdual$ as~${\rm Gal}(\Qbar/\Q)$-modules,
on applying~$\pi''$ we find that 
$E'[r] = F'[r]$ and both are
direct summands 
of~$E' \cap F'$ as~${\rm Gal}(\Qbar/\Q)$-modules''. Since it may
not be true that $\pi''(\Edual[r]) = E'[r]$ or
$\pi''(\Fdual[r]) = F'[r]$, our claim was not justified.
The claim does 
hold however, by
Lemma~\ref{lem:containment2}, under the extra hypothesis
that 
$f$ and~$g$ are not congruent modulo
a prime ideal over~$p$ to any other newforms of level dividing~$N$
for Fourier coefficients of index coprime to~$Np$
(note that in the proof of Lemma~\ref{lem:containment2}, 
the analytic or
Mordell-Weil ranks of~$f$
and~$g$ do not  play any role).
As a result, the statements of Theorem~4.4, Corollary~4.5, and Corollary~4.6
of~\cite{ag:visrk1} need
the extra hypothesis mentioned in the previous sentence
to be sure that they are valid.
\end{rmk}

\bibliographystyle{amsalpha}         


\providecommand{\bysame}{\leavevmode\hbox to3em{\hrulefill}\thinspace}
\providecommand{\MR}{\relax\ifhmode\unskip\space\fi MR }
\providecommand{\MRhref}[2]{%
  \href{http://www.ams.org/mathscinet-getitem?mr=#1}{#2}
}
\providecommand{\href}[2]{#2}

\end{document}